\documentclass{amsart}

\newif\ifdraft\draftfalse
\newif\ifcite\citefalse
\newif\ifblow\blowtrue

\usepackage{amsfonts,amssymb, amsmath}
\usepackage{oldgerm}
\usepackage{amscd}
\usepackage{eucal} % better calligraphic fonts
\usepackage{mathrsfs} % for fancy calligraphic fonts
\usepackage[hyperindex]{hyperref}  % To Make ArXiv happy
\usepackage[alphabetic]{amsrefs}
\ifdraft\ifcite\usepackage{showkeys}\else\usepackage[notcite,notref]{showkeys}\fi\fi

\usepackage{xypic}
\newdir{ >}{{}*!/-5pt/\dir{>}}

{Proposition}
\newtheorem{theorem}
{Theorem}

\newtheorem{lemma}
{Lemma}

\theoremstyle{remark}

\theoremstyle{definition}
\newtheorem{definition}
{Definition}

\newtheorem{para}[equation]{}

\theoremstyle{remark}

\newtheorem{remark}
{Remark}

\def\bz{\mathbb Z}

\def\br{\mathbb R}

\def\bp{\mathbb P}

\def\bean{\begin{eqnarray}}
\def\eean{\end{eqnarray}}
\def\bea{\begin{eqnarray*}}
\def\eea{\end{eqnarray*}}

\def\x0{X_{s_0}}
\def\2p{\bp^1\times \bp^1}

\def\a{\alpha}

\def\nm{\nonumber}
\def\er{\eqref}

\def\vol{d\mathrm{vol}}

\def\ben{\begin{equation}}
\def\een{\end{equation}}
\def\fo{\mathfrak{o}}

\def\fs{\mathfrak{s}}

\newcommand\codim{\mathop{\mathrm{codim}}\nolimits}

\newcommand\ind{\mathop{\mathrm{Ind}}\nolimits}

\newcommand\pmv{\partial_- V}

\newcommand\wt{\widetilde}

\begin{document}
\title{Secondary Chern-Euler class for general submanifold}
\author{Zhaohu Nie}
\email{znie@psu.edu}
\address{Department of Mathematics\\
Penn State Altoona\\
3000 Ivyside Park\\
Altoona, PA 16601, USA}

\date{\today}

\begin{abstract}
We define and study the secondary Chern-Euler class for a general submanifold of a Riemannian manifold. Using this class, we define and study the index for a vector field with non-isolated singularities on a submanifold. As an application, our studies give conceptual proofs of a result of Chern. 
\end{abstract}

\maketitle
% \tableofcontents

The objective of this paper is to define, study and use the 
secondary Chern-Euler class for a general submanifold of a Riemannian manifold. 
We give the definition in Definition \ref{2nd-ce} in Section 1. In Section 2, we study cohomologically the class and its relation with several other natural homology and cohomology classes.
The cases when the codimension of the submanifold is one or greater than one are different, and we consider both cases.
In Section 3, we use the secondary Chern-Euler class to define the index for a vector field with non-isolated singularities on a submanifold in Definition \ref{index}. To this end, we develop the 
notion of blow-up of the submanifold along the vector field. We then obtain three formulas in Theorem \ref{formulas} to compute  the index. Our studies, in particular, give three conceptual proofs of a classical result of Chern \cite[(20)]{chern2} concerning the paring of the secondary Chern-Euler class with the normal sphere bundle of the submanifold. Two of the proofs are given in Section 2, while the third in Section 3. 

The author thanks the referee for pointing out an oversimplifying assumption in an earlier version of the paper. 

\section{Secondary Chern-Euler class for a general submanifold}

Let $X$ be a connected oriented compact Riemannian manifold of dimension $n$. (Throughout the paper, $n=\dim X$.) The Gauss-Bonnet theorem  (see, e.g., \cite[(9)]{chern}) asserts that 
\ben
\label{gauss bonnet}
\displaystyle\int_X \Omega=\chi(X),
\een
where $\chi(X)$ is the Euler characteristic of $X$, and $\Omega$ is the Euler curvature form defined as follows. Choose local positively oriented orthonormal frames $\{e_1,\cdots,e_n\}$ of the tangent bundle $TX$. Let $(\omega_{ij})$ and 
$(\Omega_{ij})$ be the $\fs\fo(n)$-valued connection forms and curvature forms
for the Levi-Civita connection $\nabla$ of the Riemannian metric on $X$ defined by 
\begin{gather}
\nabla e_i=\sum_{j=1}^n \omega_{ij}e_j,\label{def-omega}\\
\Omega_{ij}=d\omega_{ij}-\sum_{k=1}^n \omega_{ik}\omega_{kj}.\label{def-O}
\end{gather}
Then \cite[(10)]{chern2} defines the degree $n$ form 
\ben
\label{omega}
\Omega=
\begin{cases}
0, & n\text{ odd,}\\
\displaystyle(-1)^{m}\frac{1}{2^{2m}\pi^m m!}\sum_{i} \epsilon(i)\Omega_{i_1i_2}\cdots\Omega_{i_{n-1}i_n}, & n=2m \text{ even},
\end{cases}
\een
where the summation runs over all permutations $i$ of $\{1,2,\cdots,n\}$, and $\epsilon(i)$ is the sign of $i$. 
$\Omega$ does not depend on the choice of local frames. 

Chern \cite[(9)]{chern2} defines on the unit sphere bundle $STX$, consisting of unit vectors in the tangent bundle $TX$, a form $\Phi$ (called $\Pi$ in Chern's papers) of degree $(n-1)$ as follows. 
To a unit tangent vector $v\in STX$, we attach a local positively oriented orthonormal frame $\{e_1,\cdots,e_{n-1},e_n\}$ such that $v=e_n$. 
For $k=0,1,\cdots,[\frac{n-1}2]$, define
\ben\label{phi_j}
\Phi_k=\sum_{\alpha} \epsilon(\alpha) \Omega_{\alpha_{1}\alpha_{2}}\cdots \Omega_{\alpha_{2k-1}\alpha_{2k}}\omega_{\alpha_{2k+1}n}\cdots\omega_{\alpha_{n-1}n},
\een
where the summation runs over all permutations $\alpha$ of $\{1,2,\cdots,n-1\}$. Also define
\begin{gather}
\Phi=\frac{1}{(n-2)!!|S^{n-1}|} \sum_{k=0}^{[\frac{n-1}2]} (-1)^k \frac{1}{2^k k! (n-2k-1)!!} \Phi_k,
\label{phi}
\end{gather}
where 
\ben
\label{sphere}
|S^{n-1}|=\begin{cases}
 \frac{(2\pi)^m}{(n-2)!!}, & n=2m \text{ even},\\
  \frac{2(2\pi)^{m}}{(n-2)!!}, & n=2m+1\text{ odd}
\end{cases}
\een
is the 
surface area of the unit $(n-1)$ sphere. 
The $\Phi_k$ and hence $\Phi$ do not depend on the choice of $e_1,\cdots,e_{n-1}$. Note in particular that the zeroth term
\ben
\label{0th term}
\wt{\Phi_0}=\frac{1}{(n-2)!!|S^{n-1}|} \frac 1{(n-1)!!}\sum_\alpha \epsilon(\alpha)\omega_{\alpha_{1}n}\cdots\omega_{\alpha_{n-1}n}
=\wt{\vol}_{n-1}
\een
is the unit volume form when restricted to a fiber sphere $ST_xX$ for $x\in X$. 

Then 
\cite[(11)]{chern2} proves that 
\ben
\label{dphi}
d\Phi=-\Omega.
\een
The $\Phi$, the Stokes' theorem, \er{0th term} and the Poincar\'e-Hopf theorem were  then used in \cite[(18)]{chern2} and \cite[(25)]{chern} to give an intrinsic proof for the Gauss-Bonnet theorem \er{gauss bonnet}.

\begin{definition}\label{2nd-ce} Let $M\subset X$ be a connected oriented submanifold of dimension $m$.
Then $\Phi$  in \er{phi} is a closed form when  restricted to $STX|_M$ in view of \er{dphi} and \er{omega}, since even if $n$ is even, $\Omega|_M=0$ by dimensional reasons. 

We call the restriction of $\Phi$ to $STX|_M$ the \emph{secondary Chern-Euler form of $M$ in $X$}, and its cohomology class $[\Phi]\in H^{n-1}(STX|_M,\br)$ \emph{the secondary Chern-Euler class}. 
\end{definition}

\section{Cohomological studies}

In this section, we first prove the following
\begin{theorem} There are three cases concerning the secondary Chern-Euler class $[\Phi]\in H^{n-1}(STX|_M,\br):$
\begin{enumerate}
\item When $\mathrm{codim}\, M\geq 2$, $[\Phi]\in H^{n-1}(STX|_M,\bz)$ is integral and independent of the choice of the connection 
%$(\omega_{ij})$ 
and hence of the Riemannian metric on $X$;
\item When $\mathrm{codim}\, M=1$ and $\dim X$ is odd, $[\Phi]\in H^{n-1}(STX|_M,\bz)\otimes \frac{1}{2}$ is half-integral and independent of the choice of the connection; 
\item When $\mathrm{codim}\, M=1$ and $\dim X$ is even, $[\Phi]\in H^{n-1}(STX|_M,\br)$ is only real and depends on the connection.
\end{enumerate}
\end{theorem}

\begin{proof}
%We start by studying cohomologically $H^{n-1}(STX|_M,\bz)$. 
% and the relation of $[\Phi]$ with several other natural homology and cohomology classes.
We prove the theorem by computing the integrals of $\Phi$ over generators of $H_{n-1}(STX|_M,\bz)$. Consider the following cofibration sequence
\ben\label{cofib}
STX|_M\to DTX|_M\to Th(TX)|_M,
\een
where $DTX|_M$ is the unit disk bundle over $M$, which is homotopic to $M$, and $Th(TX)|_M$ is the Thom space of $TX$ restricted to $M$. 
This gives rise to an exact sequence
\begin{gather*}
0 \to H_n(Th(TX)|_M,\bz)\to  H_{n-1}(STX|_M,\bz)\to H_{n-1}(M,\bz)\to 0
\end{gather*}
when $n\geq 2$, which is obviously the only interesting case.
Therefore one has
 \ben
 \label{1 dim}
 H_{n-1}(STX|_M,\bz)\cong 
\begin{cases}
 \bz & \mathrm{codim}\,M\geq 2,\\
 \bz\oplus \bz & \mathrm{codim}\,M=1.
 \end{cases}
 \een 

One generator of $H_{n-1}(STX|_M,\bz)$ is a fiber sphere $ST_x X$, for $x\in M$. (For simplicity of notation, we do not 
distinguish cycles from the homology classes, although we do distinguish closed forms from the cohomology classes.) 
The paring of the secondary Chern-Euler class  $[\Phi]\in H^{n-1}(STX|_M)$ with $ST_xX$ is  
\ben
\label{pd}
\int_{ST_x X}\Phi=1,
\een 
by \er{0th term} since all curvature forms vanish when restricted to $x$. 

\er{1 dim} and \er{pd} imply case (a). 

%$[\Phi]$ is integral and does not depend on the metric, in the case of $\mathrm{codim}\,M\geq 2$.  

When $\mathrm{codim}\,M=1$, let $\vec n$ denote the unit normal vector field of $M$ such that for a positively oriented frame $\{e_1,\cdots,e_{n-1}\}$ of $TM$, $\{e_1,\cdots,e_{n-1},\vec n\}$ is a positively oriented frame of $TX$. Then the image $\vec n(M)$, which we will call $M^+$, defines the other generator of $H_{n-1}(STX|_M,\bz)$ in \er{1 dim}. 

\er{1 dim}, \er{pd} and Lemma \ref{do it} below imply case (b). 

A simple example like a general circle $S^1$ on a sphere $S^2$ shows case (c). In this example, the relative Gauss-Bonnet theorem \cite[(19)]{chern2} asserts
\begin{equation}
\label{pair M}
\int_{(S^1)^+} \Phi=\chi(D)-\int_{D}\Omega,
\end{equation}
where $D$ is the region of $S^2$ bounded by $S^1$ such that $\vec n$ points outward to $D$. Therefore $\chi(D)=1$, but $\int_D\Omega$ takes real values and depends on the metric. 
\end{proof}

\begin{lemma}\label{do it} When $\mathrm{codim}\,M=1$ and $\dim X$ is odd, 
\ben\label{odd}
\int_{M^+} \Phi=\frac 1 2\chi(M).
\een
\end{lemma}

\begin{proof}
This can be seen in two ways. 

A direct way, using ideas from Chern \cite{chern2}, is by showing 
\ben
\label{1/2}
\Phi|_{M^+}=\frac 1 2 \tilde\Omega,
\een
where $\tilde\Omega$ is the pullback of the Euler curvature form of $M$ for the induced metric to $M^+\subset STX|_M$. Then the Gauss-Bonnet theorem \er{gauss bonnet} gives \er{odd}. \er{1/2} is proved as follows. 

Choose local frames $\{e_1,\cdots,e_{n-1},e_n\}$ for $TX|_M$ such that $\{e_1,\cdots,e_{n-1}\}$ are local frames for $TM$ and $e_n=\vec n$ . Assume that $\dim X=n=2m+1$. Then by \er{phi},  \er{phi_j} and \er{sphere}, one has
\begin{align*}
\Phi|_{M^+}&=\frac{1}{2^{m+1}\pi^m}\sum_{k=0}^m \frac{(-1)^k}{2^k k! (2m-2k)!!}\sum_\alpha \epsilon(\a)\Omega_{\alpha_1\alpha_2}\cdots\Omega_{\alpha_{2k-1}\alpha_{2k}}\omega_{\alpha_{2k+1}n}\cdots\omega_{\alpha_{2m}n}\\
&=\frac 1 {2^{2m+1}\pi^m}\sum_{k=0}^m \frac{(-1)^k}{k! (m-k)!}\sum_\alpha \epsilon(\a)\Omega_{\alpha_1\alpha_2}\cdots\Omega_{\alpha_{2k-1}\alpha_{2k}}\omega_{\alpha_{2k+1}n}\cdots\omega_{\alpha_{2m}n},
\end{align*}
where one uses $(2m-2k)!!=2^{m-k}(m-k)!$. 

Using \er{def-O}, one has
$$
\tilde \Omega_{\alpha\beta}=\Omega_{\alpha\beta}-\omega_{\alpha n}\omega_{\beta n}
$$
for $1\leq \alpha,\beta\leq n-1=2m$, where the $\Omega_{\a\beta}$ and $\tilde\Omega_{\a\beta}$ are the curvature forms of $X$ and $M$.  
Therefore by \er{omega} and multinomial theorem, one has
\begin{align*}
\tilde \Omega&=\frac{(-1)^{m}}{2^{2m}\pi^m m!}\sum_{\alpha} \epsilon(\alpha)(\Omega_{\alpha_1\alpha_2}-\omega_{\alpha_1 n}\omega_{\alpha_2 n})\cdots(\Omega_{\alpha_{2m-1}\alpha_{2m}}-\omega_{\alpha_{2m-1} n}\omega_{\a_{2m} n})\\
&=\frac{(-1)^{m}}{2^{2m}\pi^m m!}\sum_{k=0}^m \frac{m!}{k!(m-k)!}(-1)^{m-k}\sum_\alpha \epsilon(\a)\Omega_{\alpha_1\alpha_2}\cdots\Omega_{\alpha_{2k-1}\alpha_{2k}}\omega_{\alpha_{2k+1}n}\cdots\omega_{\alpha_{2m}n}\\
&=\frac{1}{2^{2m}\pi^m}\sum_{k=0}^m \frac{(-1)^k}{k!(m-k)!} \sum_\alpha \epsilon(\a)\Omega_{\alpha_1\alpha_2}\cdots\Omega_{\alpha_{2k-1}\alpha_{2k}}\omega_{\alpha_{2k+1}n}\cdots\omega_{\alpha_{2m}n}.
\end{align*}
Direct comparison gives \er{1/2}. 

Another way follows from Remark \ref{codim-1}, which says that in this case \er{odd} is equivalent to \er{chern's}, and our Proofs 1 and 3 of Theorem \ref{main thm}, which work in all codimensions. 
\end{proof}

\begin{remark} Case (a) is actually a manifestation of a general phenomenon for Chern-Simons forms as stated in \cite[Thm 3.9, Cor 3.17]{cs}. 
\end{remark}

\begin{remark}\label{total geod} In general, when $\codim M=1$ and $\dim X$ is even, $\int_{M^+} \Phi$ stands for the geodesic curvature of $M$. It vanishes for a totally geodesic submanifold $M$, since then all the $\omega_{\alpha n}=0$ by \er{def-omega} and the total geodesicity of $M$ (recall that we choose $e_n=\vec n$), but each summand of all the $\Phi_k$ in \er{phi_j} contains at least one $\omega_{\alpha n}$ since it is of an odd degree $n-1$. 
\end{remark}

\begin{remark} The cohomology $H^{n-1}(STX|_M,\bz)$ and the class $[\Phi]\in H^{n-1}(STX|_M)$ 
have already
been studied in \cite{sha} for $M=\partial X$, under the condition that the metric on $X$ is \emph{locally product} near $M$. 
This in particular means that $M$ is a totally geodesic submanifold of $X$. Therefore \cite[pp 1156 Special Cases]{sha} are special cases of our cases (b) and (c), in view of Remark \ref{total geod}.
\end{remark}

We now study the relation of $[\Phi]$ with some other natural homology and cohomology classes. From \er{cofib}, one has the following dual homomorphisms
\begin{gather}
%\label{del}
\delta: H^{n-1}(STX|_M,\br)
%\overset{\cong}
{\to} H^n(Th(TX)|_M,\br),\nm\\
\label{del'}
\delta': H_n(Th(TX)|_M,\br)
%\overset{\cong}
{\to} H_{n-1}(STX|_M,\br).
\end{gather}
One then has 
\ben
\label{phi & thom}
\delta[\Phi]=\gamma_{TX},
\een
where $\gamma_{TX}\in H^n(Th(TX)|_M,\bz)\cong \bz$ is the Thom class of $TX|_M$. To see this, note that 
in \er{del'} by definition
\ben
\label{hom conn}
\delta'(DT_xX/ST_xX)=ST_xX,
\een
where $DT_xX/ST_xX\in H_n(Th(TX)|_M,\bz)\cong \bz$ is a generator dual to the Thom class. \er{phi & thom} then follows from
$$
\int_{DT_xX/ST_xX}\delta[\Phi]=\int_{\delta'(DT_xX/ST_xX)}[\Phi]=\int_{ST_xX}[\Phi]=1
$$
by \er{pd}, where we denote the pairing of cohomology and homology by integration. 

\begin{para}\label{V}
Let $V$ be a generic vector field on $X$, and consider its restriction $V|_M$ on $M$. Generically, $V|_M$ has no singularities. Let $\alpha_V:M\to STX|_M$ be defined by rescaling $V|_M$, i.e., $\alpha_V(x)=\frac{V(x)}{|V(x)|},\ \forall x\in M$. 
Then $\alpha_V(M)$ is a dimension $m$ cycle in $STX|_M$, and hence defines a homology class in $H_m(STX|_M,\bz)$. 
For $x\in M$, the intersection $\alpha_V(M)\cdot ST_xX=\alpha_V(x)$ is one point. Therefore when $\codim M\geq 2$, 
\ben
\label{p.d.}
\alpha_V(M)=\text{P.D.}([\Phi])
\een
is the Poincar\'e dual of $[\Phi]$, in view of \er{pd} and \er{1 dim}.  
\end{para}

One has the decomposition
\ben
\label{decomp}
TX|_M=TM\oplus NM,
\een
where $NM$ is the normal bundle of $M$ in $X$. 
The normal sphere bundle $SNM$, consisting of unit normal vectors, defines another homology class in $H_{n-1}(STX|_M,\bz)$. 

As application of our cohomological studies, we get two conceptual proofs
of the following result about the pairing between $[\Phi]$ and $SNM$, which was first proved in \cite{chern2} in a computational way. 

\begin{theorem}\cite[(20)]{chern2} 
\label{main thm}
One has
\ben
\label{chern's}
\int_{SNM} \Phi=\chi(M).
\een
\end{theorem}

%\begin{remark} There seems to be some sign ambiguity between our statement and Chern's in \cite[(20)]{chern2}. Actually our definition of $\Phi$ in \er{phi} differs from that of Chern's in \cite[(9)]{chern2} by a negative sign when $\dim X$ is odd. Also  we feel Chern made a mistake in the even case in \cite[pp 680 last line]{chern2}, where the sign should not be there. 
%\end{remark}

\begin{proof}[Proof 1 of Theorem \ref{main thm}.] %Note first that in \er{del}, Furthermore, 
Notation as before. Note that $Th(NM)$ defines a homology class in $H_n(Th(TX)|_M,\bz)$. Similar to \er{hom conn}, one has in \er{del'}, $\delta'(Th(NM))=SNM$. Therefore
% one has
$$
\int_{SNM}\Phi=\int_{\delta'(Th(NM))} \Phi=\int_{Th(NM)} \delta[\Phi]=\int_{Th(NM)} \gamma_{TX}
%=\int_M\gamma_{TM}
=\chi(M)
$$
by \er{phi & thom}. Here the last equality follows from some basic knowledge about Thom classes. In more detail, one has the following  commutative diagrams in view of \er{decomp}
$$
\xymatrix{
H^0(M)\ar[d]^{\cup\gamma_{TM}}_\cong & \\
H^m(Th(TM))\ar[r]^{i^*}\ar[d]^{\cup\gamma_{NM}}_\cong & H^m(M)\ar[d]^{\cup\gamma_{NM}}_\cong\\
H^n(Th(TX)|_M)\ar[r]^{i^*} & H^n(Th(NM));}
\quad
\xymatrix{
1\ar@{|->}[d]^{\cup\gamma_{TM}} & \\
\gamma_{TM}\ar@{|->}[r]^{i^*}\ar@{|->}[d]^{\cup\gamma_{NM}} & e_{TM}\ar@{|->}[d]^{\cup\gamma_{NM}}\\
\gamma_{TX}\ar@{|->}[r]^{i^*} & i^*\gamma_{TX},}
$$
where the $i^*$ are induced by $i:M\to Th(TM)$ and $i:Th(NM)\to Th(TX)|_M$ defined by the zero section of $TM$,
% (and omitted in previous notation), 
and $e_{TM}\in H^m(M,\bz)$ is the Euler class of $M$.  
Therefore 
$$
\int_{Th(NM)} \gamma_{TX}=\int_{Th(NM)} i^*\gamma_{TX}=\int_M e_{TM}=\chi(M).
$$
\end{proof}

\begin{remark}\label{codim-1} When $\mathrm{codim}\,M=1$, 
$SNM=M^+-M^-$ as a homology class in $H_{n-1}(STX|_M,\bz)$, where $M^-=(-\vec{n})(M)$. One has $\int_{M^-}\Phi=(-1)^n\int_{M^+}\Phi$ by an analysis of \er{phi_j} and how one chooses local frames.   
%, which can be best seen from another definition of $\Phi$ in \cite[(15)]{chern} using the same local frames for both $M^+$ and $M^-$.  
Therefore when $n=\dim X$ is even both sides of \er{chern's} are zero, and when $n$ is odd \er{chern's} is just a doubling of \er{odd}. This was also asserted in \cite[pp 682]{chern2}. 
Therefore when $\codim M=1$ and $\dim X$ is even, the two statements \er{odd} and \er{chern's} are equivalent. Hence Lemma \ref{do it} implies this case of Theorem \ref{main thm} and vice versa. So our Proofs 1 and 3 of Theorem \ref{main thm}, which work in all codimensions, also imply Lemma \ref{do it}. However, the following Proof 2 only works when $\codim M\geq 2$. 
%In \emph{the following two proofs} we will only be concerned with the higher codimension case, and our above studies apply well.
\end{remark}

\begin{proof}[Proof 2 of Theorem \ref{main thm} when $\codim M\geq 2$.] 
Continue to work with the vector field $V$ introduced in \ref{V}. Consider the  projection $\partial V$ of $V|_M$ to $TM$. Generically, $\partial V$ has only isolated singularities 
with indices $\pm 1$, and the sum of its indices is $\chi(M)$ by the Poincar\'e-Hopf theorem. 
Suppose $p$ is a singular point of $\partial V$. Then $V$ is perpendicular to $M$ at $p$, and hence $\alpha_V(p)\in SNM$. Since $\ind_p\partial V=\pm 1$, it can be seen that one has transversal intersection $\alpha_V(M)\pitchfork_{\alpha_V(p)} SNM$. Furthermore, the intersection index $\iota_{\alpha_V(p)}(\alpha_V(M),SNM)=(-1)^m\ind_p \partial V$, where $m=\dim M$, for suitable orientations. 
Therefore the intersection number $\#(\alpha_V(M),SNM)=(-1)^m \chi(M)$,
which implies by \er{p.d.}
$$ 
\int_{SNM}\Phi=\#(\alpha_V(M),SNM)=(-1)^m\chi(M)=\chi(M),
$$
where the last equality holds by the obvious reason that $\chi(M)=0$ when $m$ is odd.
\end{proof}

\section{Index of vector field with non-isolated singularities}

In this section, we first use the secondary Chern-Euler class to define the index for a vector field with non-isolated singularities. This also involves the notion of blow-up of a submanifold along a vector field $V$ which vanishes on it. 
%As application, we give a third proof of Theorem \ref{main thm}. 

Let $V$ be a vector field on $X$ with non-isolated singularities on a submanifold $M$. 
%Let $B_r(M)$ and $S_r(M)$ be as in Proof 3. 
Let 
%$M\subset  U\subset X$ 
$U$ be a closed neighborhood of $M$ in $X$, and suppose its boundary $\partial U$ is smooth.
%  is smooth. 
Assume that $V$ has no singularities on $U-M$. Consider $\alpha_V:U-M\to STX$ by rescaling $V$. 

\begin{definition}\label{index} The closure of the image $\overline{\alpha_V(U-M)}$ defines a homology class in $H_n(STX|_{U},STX|_{\partial U}\cup STX|_M,\bz)$. Under the connecting homomorphism for relative homology 
$$
\partial:H_n(STX|_{U},STX|_{\partial U}\cup STX|_M,\bz)\to H_{n-1}(STX|_{\partial U}\cup STX|_M,\bz),
$$
one has
\ben
\label{boundary}
\partial(\overline{\alpha_V(U-M)})=\alpha_V(\partial U)-Bl_V(M),
\een
where $Bl_V(M)\in H_{n-1}(STX|_M,\bz)$ is defined by the above. We call it the \emph{blow-up of $M$ along $V$} as a homology class. The \emph{index of $V$ at $M$}
%, as an integer, 
is defined by
\ben
\label{def ind}
\ind_M V=\int_{Bl_V(M)}\Phi,
%\lim_{r\to 0}\int_{S_r(M)} \alpha^*\Phi
\een
%for $r$ sufficiently small, 
%is the local index for $V$ on $M$.
where $[\Phi]\in H^{n-1}(STX|_M)$ is the secondary Chern-Euler class .  
\end{definition}

The following theorem gives three ways of evaluating the index $\ind_M V$. In particular, it shows that $\ind_M V$ is \emph{always} an integer independent of the metric. 

\begin{theorem}
\label{formulas}
%Notation as before. 
\begin{enumerate}
\item In terms of the Euler curvature form and the secondary Chern-Euler class, one has
\ben\label{loc ind}
\ind_M V=\int_{\alpha_V(\partial U)} \Phi+\int_{U}\Omega.
\een
\item 
Extend $V|_{\partial U}$ to a 
%different 
vector field $\tilde V$ on $U$ with isolated singularities. Then
\ben
\label{wiggle}
\ind_M V=\ind\tilde V,
\een
where $\ind\tilde V$ denotes the sum of local indices of $\tilde V$ at its isolated singularities. 
\item Let $\partial V$ denote the tangential projection of $V$ to the tangent space of the boundary $\partial U$, and $\partial_- V$
% (resp. $\partial_+ V$) 
the restriction of $\partial V$ to the parts where $V$ points inward 
%(resp. outward) 
to $U$. Then generically one has
\ben\label{law} 
\ind_M V=\chi(U)-\ind\partial_- V,
\een
\end{enumerate}
\end{theorem}

\begin{proof}
\begin{enumerate}
\item 
Following  \cite[(25)]{chern}, and using \er{dphi}, Stokes' theorem, and \er{boundary}, 
%to $\alpha_V$, 
one has
\begin{equation}
\label{new one}
\int_{U} -\Omega=\int_{\overline{\alpha_V(U-M)}}-\Omega=\int_{\overline{\alpha_V(U-M)}}d\Phi
=\int_{\alpha_V(\partial U)} \Phi-\int_{Bl_V(M)} \Phi,
%=\int_{S_r(M)} \vec n^*\Phi-\int_{SNM} \Phi.\nm
\end{equation}
which then gives \er{loc ind} in view of \er{def ind}. 

Note that from \er{loc ind}, one immediately gets
$$
\ind_M V=\lim_{U\to M} \int_{\alpha_V(\partial U)} \Phi.
$$

%The right hand side is an integer. Actually, e
%Then the standard procedure as in \er{new one} gives 
\item 
Applying the standard procedure, as in \cite[(25)]{chern} and \er{new one}, and using \er{0th term}, one gets 
\ben\label{tv}
\ind\tilde V=\int_{\alpha_V(\partial U)} \Phi+\int_{U}\Omega.
\een
The index $\ind\tilde V$ clearly does not depend  on the extension $\tilde V$ from this formula, once one fixes the metric. Therefore
comparison of \er{loc ind} with \er{tv} gives \er{wiggle}. This also implies that $\ind_M V$ in Definition \ref{index} is an integer independent of the metric, since it is equal to $\ind \tilde V$.

\item 
From \cite{nie}, one knows that the right hand side of \er{loc ind}
$$
\int_{\a_V(\partial U)}\Phi+\int_U\Omega=\chi(U)-\ind(\pmv), 
$$
which proves \er{law}. Note that in the topology literature, \er{law} is called the Law of Vector Fields, and was first proved in \cite{morse}. 
\end{enumerate}
\end{proof}

As an application, we give a third proof of Theorem \ref{main thm}.

\begin{proof}[Proof 3 of Theorem \ref{main thm}.] 
We will apply \er{wiggle} to the radial vector field around $M$. 

For $r>0$ small, let $U=B_r(M)=\{x\in X|d(x,M)\leq r\}$ be a tubular neighborhood of $M$ in $X$. Then its boundary is $\partial U=S_r(M)=\{x\in X|d(x,M)= r\}$. For $x\in B_r(M)$, let $p(x)$ be the point on $M$ such that $d(x,p(x))=d(x,M)$. 
Choose $r$ sufficiently small so that $p(x)$ is unique and there is a unique shortest geodesic connecting $p(x)$ and $x$. 
Denote $s(x)=d(x,M)$ and treat $s$ as a coordinate on $B_r(M)$. Let  $\vec n:=\frac{\partial}{\partial s}$.
%defines a unit vector field on $B_r(M)-M$. 
Note that $\vec n(x)=\frac{\partial}{\partial s}(x)$ is the unit tangent vector at $x$ of the unique shortest geodesic starting from $p(x)$ and passing $x$. Now consider the vector field $V=\frac s r\vec n=\frac s r\frac{\partial}{\partial s}$. Then $V$ has singularities on $M$ corresponding to $s=0$, and on on $S_r(M)$ corresponding to $s=r$, $V=\vec n$ is the outward normal vector field. After rescaling, $\alpha_V=\vec n:B_r(M)-M\to STX$. 

%, and $\\bot S_{s_0}(M), \forall 0\leq s_0\leq r$. 

%By minimality, $\vec n(x)\in SNM$.  

Consider the closure of the image $\overline{\vec n(B_r(M)-M)}$ in $STX$. 
%In the present case, it is a submanifold with boundary, and 
Its boundary is
$$
%\label{blow-up?}
\partial (\overline{\vec n(B_r(M)-M)})=\vec n(S_r(M))-SNM.
$$
Therefore 
$
Bl_V(M)=SNM,
$
and 
$$
\int_{SNM}\Phi=\ind_M V=\ind \tilde V,
$$
by Definition \ref{index} and \er{wiggle}. Here $\tilde V$ is a generic extension with isolated singularities to $B_r(M)$ of the vector field $\vec n$ on $S_r(M)$.
%, which is the outward normal vector field. Then b
By definition, $\ind \tilde V=\chi(B_r(M))$. 
Therefore one is done 
by the homotopy invariance of Euler characteristic.
\end{proof}

\begin{remark} 
%None of  our proofs used the Gauss-Bonnet theorem for $M$. Instead they used some topological knowledge, the Poincar\'e-Hopf theorem, or the relative Gauss-Bonnet theorem for $B_r(M)$.  
Chern's computation in \cite{chern2} proves that $\int_{SNM}\Phi=\int_M \Omega_M$. This and Theorem \ref{main thm}, with our conceptual proofs, would prove  the Gauss-Bonnet theorem \er{gauss bonnet} for $M$, $\int_M \Omega_M=\chi(M)$, if one did not know it. Such a route was taken historically by \cites{allen,fenchel} to prove the Gauss-Bonnet theorem for a submanifold of higher codimension in an Euclidean space from the known result of a hypersurface. 
\end{remark}

\end{document}